\documentclass[a4paper,fleqn]{cas-sc}

\usepackage[numbers,sort&compress]{natbib}
\usepackage{amssymb}
\usepackage{caption}
\usepackage{tikz}

\usetikzlibrary{shapes.geometric, arrows.meta, calc, positioning, shadows, babel}

\newtheorem{theorem}{Theorem}
\newtheorem{proposition}[theorem]{Proposition}
\newtheorem{lemma}[theorem]{Lemma}
\newdefinition{rmk}{Remark}
\newproof{pf}{Proof}

\begin{document}
\let\WriteBookmarks\relax
\def\floatpagepagefraction{1}
\def\textpagefraction{.001}

\shorttitle{Exact Integration and Invertibility Completeness of the Nonlinear Relaxation Equation}    
\shortauthors{L. Mrazík}  

\title [mode = title]{Exact Integration and Invertibility Completeness of the Nonlinear Relaxation Equation}  

\author{Luk\'{a}\v{s} Mraz\'{i}k}[orcid=0000-0002-8847-5697]
\cormark[1]
\ead{mrazikl@vscht.cz}
\credit{Methodology, Formal analysis, Writing -- original draft, Writing -- review \& editing}

\affiliation{
    organization={Department of Mathematics, Informatics and Cybernetics, University of Chemistry and Technology, Prague},
    addressline={Technick\'{a} 5},
    city={Prague},
    postcode={166 28},
    country={Czechia}
}
\cortext[1]{Corresponding author}

\begin{abstract}
We examine the nonlinear initial value problem $\dot{y} = C(1 - y^\alpha)$ modeling polytropic gas permeation. For rational $\alpha$, we derive the exact primitive via cyclotomic partial fraction decomposition, establishing its real trigonometric representation. By synthesizing Ritt's theorem with Rosenlicht's differential algebraic residue criterion, we prove the exact step response $y(t)$ is an elementary function if and only if $\alpha \in \{1, 2\}$. Admitting the Lambert $W$ function extends closed-form solvability strictly to $\alpha \in \{-1, 1/2\}$. Finally, the asymptotic limit $\alpha \to 0$ converges to a singular dynamics solved by the logarithmic integral. These results constitute a complete classification of all explicit closed-form solutions admitted by this system.
\end{abstract}

\begin{keywords}
 Nonlinear ordinary differential equations \sep Cyclotomic integration \sep Differential algebra \sep Lambert $W$ function \sep Invertibility completeness
\end{keywords}

\maketitle

\section{Introduction and Physical Genesis}
In the macroscopic modeling of compressible fluid flow through rigid porous media, combining mass conservation with Darcy's law for a steady polytropic gas (index $n \in \mathbb{R} \setminus \{0, -1\}$) dictates the rate of mass accumulation in dead-end collection cells \cite{Mrazik2025, Bear1972}. Defining the exponential parameter $\alpha := (n+1)/n$ and the normalized downstream pressure $y(t) := p_{\mathrm{out}}(t)/p_{\mathrm{in}}$, the pressure dynamics reduce to the autonomous initial value problem:
\begin{equation}\label{eq:ode_main}
    \dot{y} = C (1 - y^\alpha), \quad y(0) = 0,
\end{equation}
where the constant $C > 0$ lumps matrix permeability, fluid viscosity, and geometric parameters. We restrict our analysis to the physical domain $y \in [0, 1)$ for $t \ge 0$. Although this equation is easily integrated numerically, exact closed-form inverses are crucial for rigorously benchmarking porous media simulators and directly extracting $C$ from transient experimental data.

\section{Cyclotomic Integration of the Autonomous Initial Value Problem}
Separation of variables in \eqref{eq:ode_main} leads to the integral relation
\begin{equation}\label{eq:primitive_integral}
    F_{\alpha}(y) := \int_0^y \frac{\mathrm{d}s}{1 - s^\alpha} = C t.
\end{equation}
For arbitrary real $\alpha > 0$, the geometric series expansion for $|s| < 1$ yields
\begin{equation}\label{eq:hyper_series}
    F_\alpha(y) = \sum_{m=0}^\infty \int_0^y s^{\alpha m} \, \mathrm{d}s = \sum_{m=0}^\infty \frac{y^{\alpha m + 1}}{\alpha m + 1} = y \, {}_2F_1\left(1, \frac{1}{\alpha}; 1 + \frac{1}{\alpha}; y^\alpha\right),
\end{equation}
where ${}_2F_1$ is the Gauss hypergeometric function \cite{NIST:DLMF}. 

\begin{rmk}[Symmetry of Negative Exponents]
Series convergence requires $\alpha > 0$, but analytical generality is preserved for $\alpha < 0$. Setting $\beta := -\alpha$ yields the algebraic identity $(1 - s^{-\beta})^{-1} = 1 - (1 - s^\beta)^{-1}$, which integrates to $F_{-\beta}(y) = y - F_{\beta}(y)$. Because the primitives differ strictly by the elementary linear term $y$, their differential algebraic properties and invertibility classification remain identical.
\end{rmk}

When $\alpha$ is rational, this transcendental series reduces to finite algebraic and logarithmic terms.

\begin{proposition}[Cyclotomic Decomposition]\label{prop:cyclotomic}
Let $\alpha = A/B \in \mathbb{Q}_{>0}$ with coprime integers $A, B \in \mathbb{N}$. The primitive $F_{\alpha}(y)$ is given in closed form via the substitution $u = y^{1/B}$ by
\begin{equation}\label{eq:decomp_complex}
    F_{\alpha}(y):= \int \frac{\mathrm{d}y}{1 - y^{A/B}} = -\sum_{k=1}^{\lfloor (B-1)/A \rfloor} \frac{B \, u^{B - kA}}{B - kA} - \frac{B}{A} \sum_{k=0}^{A-1} \omega_k^B \operatorname{Log}(u - \omega_k) + C_0,
\end{equation}
where $\omega_k = \exp(2\pi i k / A)$ are the distinct complex $A$-th roots of unity, $\operatorname{Log}$ denotes the principal complex logarithm, and $C_0$ is an integration constant enforcing $F_{\alpha}(0) = 0$.
\end{proposition}

\begin{pf}
Substituting $u = y^{1/B}$ implies $y = u^B$ and $\mathrm{d}y = B u^{B-1} \mathrm{d}u$. The integral transforms into:
\begin{equation}
    \int \frac{\mathrm{d}y}{1 - y^{A/B}} = \int \frac{B u^{B-1}}{1 - u^A} \, \mathrm{d}u.
\end{equation}
If $B - 1 \ge A$, we perform Euclidean polynomial division. Defining the remainder $r := (B-1) \bmod A$, the integrand decomposes into a~polynomial and a~proper rational function:
\begin{equation}
    \frac{B u^{B-1}}{1 - u^A} = -B \sum_{k=1}^{\lfloor (B-1)/A \rfloor} u^{B - 1 - kA} + \frac{B u^r}{1 - u^A}.
\end{equation}
Direct term-by-term integration of this polynomial part yields the first sum in \eqref{eq:decomp_complex}.

The remaining proper rational function possesses $A$ simple poles located precisely at the cyclotomic points $\omega_k = \exp(2\pi i k / A)$ for $k \in \{0, 1, \dots, A-1\}$. The partial fraction expansion over $\mathbb{C}(u)$ takes the form:
\begin{equation}
    \frac{B u^r}{1 - u^A} = \sum_{k=0}^{A-1} \frac{c_k}{u - \omega_k}.
\end{equation}
The complex residues $c_k \in \mathbb{C}$ at these simple poles are evaluated analytically via l'H\^opital's rule:
\begin{equation}
    c_k = \lim_{u \to \omega_k} (u - \omega_k) \frac{B u^{B-1}}{1 - u^A} = \frac{B \omega_k^{B-1}}{-A \omega_k^{A-1}} = -\frac{B}{A} \omega_k^{B-A} = -\frac{B}{A} \omega_k^B,
\end{equation}
where we utilize the identity $\omega_k^A = 1$. Integrating the resulting simple fractions, $\int c_k (u - \omega_k)^{-1} \, \mathrm{d}u = c_k \operatorname{Log}(u - \omega_k)$, yields the logarithmic sum in \eqref{eq:decomp_complex}.
\end{pf}

The complex logarithmic sum from Proposition~\ref{prop:cyclotomic} contains imaginary components. Because the polynomial $1 - u^A$ has strictly real coefficients, its complex roots exist in predictable symmetric arrangements on the unit circle. We can exploit this geometry to eliminate all imaginary components based on the parity of $A$, as illustrated in Figure~\ref{fig:strategy}. Let $L(u)$ denote this logarithmic sum:
\begin{equation}
    L(u) = -\frac{B}{A} \sum_{k=0}^{A-1} \omega_k^B \operatorname{Log}(u - \omega_k).
\end{equation}

\begin{lemma}[Conjugate Symmetry]\label{lem:conjugate_pairing}
The sum $L(u)$ can be reduced to a~strictly real representation:
\begin{equation}
    L(u) = -\frac{B}{A} \sum_{k=0}^{A-1} \left[ C_k(u) + S_k(u) \right],
\end{equation}
where the real functions $C_k$ and $S_k$ are defined via the argument $\varphi_k = 2\pi k / A$ as:
\begin{align}
    C_k(u) &= \frac{1}{2} \cos(B\varphi_k) \ln(u^2 + 1 - 2u\cos\varphi_k), \label{eq:C_k} \\
    S_k(u) &= -\sin(B\varphi_k) \operatorname{atan2}(-\sin\varphi_k, u - \cos\varphi_k). \label{eq:S_k}
\end{align}
\end{lemma}

\begin{pf}
By utilizing Euler's formula, $\omega_k = \cos\varphi_k + i\sin\varphi_k$, we decompose the complex logarithm into its real and imaginary components, $\operatorname{Log}(u - \omega_k) = \ln|u - \omega_k| + i \arg(u - \omega_k)$. The squared modulus inside the real logarithm evaluates algebraically to $|u - \omega_k|^2 = (u - \cos\varphi_k)^2 + \sin^2\varphi_k = u^2 + 1 - 2u\cos\varphi_k$.

Because the roots of unity occur in complex conjugate pairs, $\omega_{A-k} = \overline{\omega_k}$, their corresponding angles satisfy $\varphi_{A-k} = 2\pi - \varphi_k \equiv -\varphi_k$. Consequently, the trigonometric components of the residues exhibit specific symmetries: $\cos(B\varphi_{A-k}) = \cos(B\varphi_k)$ and $\sin(B\varphi_{A-k}) = -\sin(B\varphi_k)$. Pairing the $k$-th and $(A-k)$-th terms in the summation analytically cancels out all imaginary components, leaving exactly twice the real parts. The real part of the logarithm contributes \eqref{eq:C_k}, while the imaginary phase argument contributes the arctangent \eqref{eq:S_k}.
\end{pf}

\begin{lemma}[Real Form for Odd Denominators]\label{lem:odd_denominators}
For odd  $A = 2M + 1$ ($M \in \mathbb{N}$), the sum evaluates to:
\begin{equation}
    L(u) = -\frac{B}{A} \ln(1 - u) - \frac{B}{A} \sum_{k=1}^M \cos(B\varphi_k) \ln(u^2 + 1 - 2u\cos\varphi_k)  + \frac{2B}{A} \sum_{k=1}^M \sin(B\varphi_k) \arctan\frac{u - \cos\varphi_k}{\sin\varphi_k} + \lambda_0,
\end{equation}
where $\lambda_0 = \frac{\pi B}{A} \cot\left(\frac{\pi B}{A}\right)$.
\end{lemma}

\begin{pf}
For $A = 2M+1$, the single real root occurs at $k=0$. Evaluating \eqref{eq:C_k} and \eqref{eq:S_k} at $\varphi_0 = 0$ isolates this root as $C_0 = \ln(1-u)$ and $S_0 = 0$. The remaining $2M$ roots are paired symmetrically across the real axis as $\omega_k$ and $\omega_{A-k}$, see Figure~\ref{fig:strategy}.

Due to this vertical reflection symmetry, $C_{A-k} = C_k$ and $S_{A-k} = S_k$. Grouping these pairs reduces the sum range to $M$ and doubles the coefficients. Finally, evaluating the absolute term required to enforce $L(0) = 0$ produces the geometric shift $\lambda_0$.
\end{pf}

\begin{lemma}[Real Form for Even Denominators]\label{lem:even_denominators}
For even $A = 2M$ ($M \in \mathbb{N}$), the sum evaluates to:
\begin{equation}
    \begin{aligned}
        L(u) ={}& \frac{2B}{A} \operatorname{artanh}(u) + \frac{2B}{A} \sum_{k=1}^{\lfloor (M-1)/2 \rfloor} \cos(B\varphi_k) \operatorname{artanh}\frac{2u\cos\varphi_k}{u^2+1} - \frac{2B}{A} \sum_{k=1}^{\lfloor (M-1)/2 \rfloor} \sin(B\varphi_k) \arctan\frac{2u\sin\varphi_k}{u^2-1} \\
        & + \frac{2B}{A} (-1)^{\frac{B-1}{2}} \arctan(u) \mathbb{1}_{\{2 \mid M\}} +\lambda_0,
    \end{aligned}
\end{equation}
where $\lambda_0=- \frac{\pi B}{A} \cot\frac{\pi B}{A}$ and $\mathbb{1}_{\{\cdot\}}$ denotes the indicator function.
\end{lemma}

\begin{pf}
For even denominators, every root $\omega_k$ possesses a~geometric antipode $\omega_{k+M} = -\omega_k$ across the origin, see Figure \ref{fig:strategy}. Pairing these antipodal roots transforms the logarithmic differences directly into hyperbolic functions:
\begin{equation}
    \frac{1}{2}\ln(u^2 + 1 - 2u\cos\varphi_k) - \frac{1}{2}\ln(u^2 + 1 + 2u\cos\varphi_k) = -\operatorname{artanh}\frac{2u\cos\varphi_k}{u^2+1}.
\end{equation}
Similarly, the phase arguments combine via arctangent addition identities to produce the corresponding $\arctan$ terms. The roots can be further paired across the imaginary axis as $k$ and $M-k$, leaving isolated real and pure imaginary roots depending on the parity of $M$. The pure imaginary roots generate the isolated $\arctan(u)$ term tracked by the indicator function $\mathbb{1}_{\{2 \mid M\}}$. Summing the series at $u=0$ evaluates the geometric constant to $-\frac{\pi B}{A} \cot\frac{\pi B}{A}$.
\end{pf}

\section{Classification of Elementary Inverses via Differential Algebra}

We now examine when the implicit solution $F_\alpha(y) = Ct$ can be solved for $y(t)$ within the field of elementary functions \cite{Bronstein2005}. 

\begin{theorem}[Ritt's Invertibility Theorem, \cite{Ritt1925}]\label{thm:ritt}
Let $f(z)$ be an elementary function. Its functional inverse $f^{-1}(z)$ is elementary if and only if $f$ can be written as a~finite chain of compositions of single-variable functions $f = \phi_m \,\circ\, \phi_{m-1} \,\circ\, \dots \,\circ\, \phi_1,$ where each mapping $\phi_i$ is either an algebraic function ($\phi_i \in \mathcal{A}$) or a~single transcendental elementary operation ($\phi_i \in \mathcal{T} = \{\exp, \ln\}$). Furthermore, if $\phi_{i+1} \in \mathcal{T}$, then $\phi_i \in \mathcal{A}$; that is, transcendental functions cannot be directly composed without an intervening algebraic stage.
\end{theorem}

\noindent
\begin{minipage}{\textwidth}
\centering
\newcommand{\unitCircleR}{1.0cm} 
\newcommand{\rootFontScale}{0.75}
\definecolor{rootgray}{gray}{0.5}

\tikzset{
    polecircle/.style={circle, fill, inner sep=1.1pt},
    polereal/.style={circle, draw, fill=white, inner sep=1.1pt},
    poleimag/.style={circle, draw, fill=white, inner sep=1.1pt,
        path picture={
            \draw[black, thin] (path picture bounding box.south west) -- (path picture bounding box.north east);
            \draw[black, thin] (path picture bounding box.north west) -- (path picture bounding box.south east);
        }
    },
    poletriangle/.style={regular polygon, regular polygon sides=3, draw, fill, inner sep=0.8pt},
    polesquare/.style={regular polygon, regular polygon sides=4, draw, fill, inner sep=1.0pt}
}

\newcommand{\drawFullCircle}[1]{
    \draw[thin, rootgray] (-1.25*\unitCircleR, 0) -- (1.25*\unitCircleR, 0);
    \draw[thin, rootgray] (0, -1.25*\unitCircleR) -- (0, 1.25*\unitCircleR);
    \draw[thin, rootgray, dashed] (\unitCircleR, 0) arc (0:360:\unitCircleR);
    \foreach \k in {0,...,#1} {\coordinate (W\k) at ({\k * (360 / #1)}:\unitCircleR);}
}

\newcommand{\drawSemiCircle}[1]{
    \draw[thin, rootgray] (-1.25*\unitCircleR, 0) -- (1.25*\unitCircleR, 0);
    \draw[thin, rootgray] (0, -0.1*\unitCircleR) -- (0, 1.25*\unitCircleR);
    \draw[thin, rootgray, dashed] (\unitCircleR, 0) arc (0:180:\unitCircleR);
    \foreach \k in {0,...,#1} {\coordinate (W\k) at ({\k * (360 / #1)}:\unitCircleR);}
}

\begin{tikzpicture}[>=stealth, scale=0.8]
    \def\innerFactor{0.8}
    \def\secondInnerFactor{0.7}
    \def\outerFactor{1.2}
    \def\labelPadding{2pt}

    \begin{scope}[yshift=0cm]
        \node[anchor=east, scale=\rootFontScale] at (-1.5, 0) {$A \bmod 2 = 1$};
        \def\numPoles{7}
        \pgfmathtruncatemacro{\maxK}{\numPoles-1}
        \pgfmathtruncatemacro{\midK}{(\numPoles-1)/2}
        
        \begin{scope}[xshift=0.0cm]
            \drawFullCircle{\numPoles}
            \foreach \k in {0,...,\maxK} { \node[polecircle] at (W\k) {}; }
            \node[fill=white, inner sep=\labelPadding] at (0,0) {$L$};
            \draw[->] ({360/\numPoles}:{\innerFactor*\unitCircleR}) arc [start angle={360/\numPoles}, end angle={\maxK*360/\numPoles}, radius={\innerFactor*\unitCircleR}];
        \end{scope}
        \node at (1.75, 0) {$\Longrightarrow$};
        \begin{scope}[xshift=3.5cm]
            \drawFullCircle{\numPoles}
            \foreach \k in {0,...,\maxK} { \node[polecircle] at (W\k) {}; }
            \node[fill=white, inner sep=\labelPadding] at (0,0) {$L$};
            \draw[->] ({360/\numPoles}:{\innerFactor*\unitCircleR}) arc [start angle={360/\numPoles}, end angle={\midK*360/\numPoles}, radius={\innerFactor*\unitCircleR}];
            \draw[->] ({\maxK*360/\numPoles}:{\innerFactor*\unitCircleR}) arc [start angle={\maxK*360/\numPoles}, end angle={360-\midK*360/\numPoles}, radius={\innerFactor*\unitCircleR}];
        \end{scope}
        \node at (5.25, 0) {$\Longrightarrow$};
        \begin{scope}[xshift=7cm]
            \drawFullCircle{\numPoles}
            \foreach \k in {0,...,\maxK} { \node[polecircle] at (W\k) {}; }
            \foreach \k in {1,...,\midK} {
                \pgfmathtruncatemacro{\reflectionK}{\numPoles-\k}
                \draw[->, thin, densely dashed, shorten <=3pt, shorten >=3pt] (W\reflectionK) -- (W\k);
            }
        \end{scope}
        \node at (8.75, 0) {$\Longrightarrow$};
        \begin{scope}[xshift=10.5cm]
            \drawFullCircle{\numPoles}
            \node[polecircle] at (W0) {};
            \foreach \k in {1,...,\midK} { \node[poletriangle] at (W\k) {}; }
            \node[fill=white, inner sep=\labelPadding] at (0,0) {$L$};
            \draw[->] ({360/\numPoles}:{\innerFactor*\unitCircleR}) arc [start angle={360/\numPoles}, end angle={\midK*360/\numPoles}, radius={\innerFactor*\unitCircleR}];
            \draw[->] ({360/\numPoles}:{\secondInnerFactor*\unitCircleR}) arc [start angle={360/\numPoles}, end angle={\midK*360/\numPoles}, radius={\secondInnerFactor*\unitCircleR}];
        \end{scope}
    \end{scope}

    \begin{scope}[yshift=-2.8cm]
        \node[anchor=east, scale=\rootFontScale] at (-1.5, 0) {$A \bmod 2 = 0$};
        \def\numPoles{10}
        \pgfmathtruncatemacro{\maxK}{\numPoles-1}
        \pgfmathtruncatemacro{\midK}{\numPoles/2 - 1}
        
        \begin{scope}[xshift=0.0cm]
            \drawFullCircle{\numPoles}
            \foreach \k in {0,...,\maxK} { \node[polecircle] at (W\k) {}; }
            \node[fill=white, inner sep=\labelPadding] at (0,0) {$L$};
            \draw[->] (0:{\innerFactor*\unitCircleR}) arc [start angle=0, end angle={\maxK*360/\numPoles}, radius={\innerFactor*\unitCircleR}];
        \end{scope}
        \node at (1.75, 0) {$\Longrightarrow$};
        \begin{scope}[xshift=3.5cm]
            \drawFullCircle{\numPoles}
            \foreach \k in {0,...,\maxK} { \node[polecircle] at (W\k) {}; }
            \node[fill=white, inner sep=\labelPadding] at (0,0) {$L$};
            \draw[->] (0:{\innerFactor*\unitCircleR}) arc [start angle=0, end angle={\midK*360/\numPoles}, radius={\innerFactor*\unitCircleR}];
            \draw[->] (180:{\innerFactor*\unitCircleR}) arc [start angle=180, end angle={\maxK*360/\numPoles}, radius={\innerFactor*\unitCircleR}];
        \end{scope}
        \node at (5.25, 0) {$\Longrightarrow$};
        \begin{scope}[xshift=7cm]
            \drawFullCircle{\numPoles}
            \foreach \k in {0,...,\maxK} { \node[polecircle] at (W\k) {}; }
            \foreach \k in {0,...,\midK} {
                \pgfmathtruncatemacro{\flipK}{\k+\numPoles/2}
                \draw[->, thin, densely dashed, shorten <=3pt, shorten >=3pt] (W\flipK) -- (W\k);
            }
        \end{scope}
        \node at (8.75, 0) {$\Longrightarrow$};
        \begin{scope}[xshift=10.5cm]
            \drawFullCircle{\numPoles}
            \foreach \k in {0,...,\midK} { \node[poletriangle] at (W\k) {}; }
            \node[fill=white, inner sep=\labelPadding] at (0,0) {$L$};
            \draw[->] (0:{\innerFactor*\unitCircleR}) arc [start angle=0, end angle={\midK*360/\numPoles}, radius={\innerFactor*\unitCircleR}];
            \draw[->] (0:{\secondInnerFactor*\unitCircleR}) arc [start angle=0, end angle={\midK*360/\numPoles}, radius={\secondInnerFactor*\unitCircleR}];
        \end{scope}
    \end{scope}

    \begin{scope}[yshift=-5.6cm]
        \node[anchor=east, scale=\rootFontScale] at (-1.5, 0) {$A \bmod 4 = 0$};
        \def\numEvenPoles{12}
        \pgfmathtruncatemacro{\evenMidK}{\numEvenPoles/2 - 1}
        \pgfmathtruncatemacro{\evenQrtK}{\evenMidK/2}
        
        \begin{scope}[xshift=0.0cm]
            \drawSemiCircle{\numEvenPoles}
            \foreach \k in {0,...,\evenMidK} { \node[poletriangle] at (W\k) {}; }
            \node[fill=white, inner sep=\labelPadding] at (0,0) {$L$};
            \draw[->] (0:{\innerFactor*\unitCircleR}) arc [start angle=0, end angle={\evenMidK*360/\numEvenPoles}, radius={\innerFactor*\unitCircleR}];
        \end{scope}
        \node at (1.75, 0) {$\Longrightarrow$};
        \begin{scope}[xshift=3.5cm]
            \drawSemiCircle{\numEvenPoles}
            \foreach \k in {0,...,\evenMidK} { \node[poletriangle] at (W\k) {}; }
            \node[fill=white, inner sep=\labelPadding] at (0,0) {$L$};
            \draw[->] ({360/\numEvenPoles}:{\innerFactor*\unitCircleR}) arc [start angle={360/\numEvenPoles}, end angle={\evenQrtK*360/\numEvenPoles}, radius={\innerFactor*\unitCircleR}];
            \draw[->] ({180-360/\numEvenPoles}:{\innerFactor*\unitCircleR}) arc [start angle={180-360/\numEvenPoles}, end angle={180-\evenQrtK*360/\numEvenPoles}, radius={\innerFactor*\unitCircleR}];
        \end{scope}
        \node at (5.25, 0) {$\Longrightarrow$};
        \begin{scope}[xshift=7cm]
            \drawSemiCircle{\numEvenPoles}
            \foreach \k in {0,...,\evenMidK} { \node[poletriangle] at (W\k) {}; }
            \foreach \k in {1,...,\evenQrtK} {
                \pgfmathtruncatemacro{\evenMirrorK}{\evenMidK+1-\k}
                \draw[->, thin, densely dashed, shorten <=3pt, shorten >=3pt] (W\evenMirrorK) -- (W\k);
            }
        \end{scope}
        \node at (8.75, 0) {$\Longrightarrow$};
        \begin{scope}[xshift=10.5cm]
            \drawSemiCircle{\numEvenPoles}
            \node[poletriangle] at (W0) {};
            \pgfmathtruncatemacro{\topK}{\evenQrtK+1}
            \node[poletriangle] at (W\topK) {};
            \foreach \k in {1,...,\evenQrtK} { \node[polesquare] at (W\k) {}; }
            \node[fill=white, inner sep=\labelPadding] at (0,0) {$L$};
            \draw[->] ({360/\numEvenPoles}:{\innerFactor*\unitCircleR}) arc [start angle={360/\numEvenPoles}, end angle={\evenQrtK*360/\numEvenPoles}, radius={\innerFactor*\unitCircleR}];
            \draw[->] ({360/\numEvenPoles}:{\secondInnerFactor*\unitCircleR}) arc [start angle={360/\numEvenPoles}, end angle={\evenQrtK*360/\numEvenPoles}, radius={\secondInnerFactor*\unitCircleR}];
        \end{scope}
    \end{scope}

\end{tikzpicture}

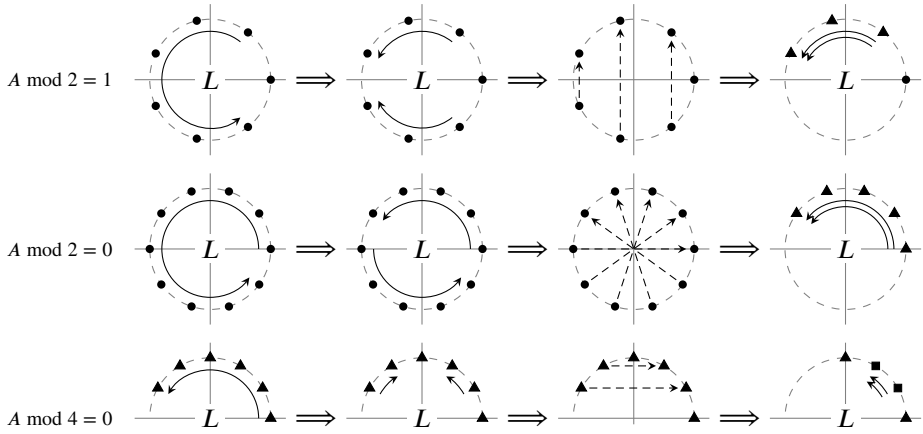
\captionof{figure}{Progression of geometric pairing strategies for cyclotomic roots based on denominator parity.}
\label{fig:strategy}
\end{minipage}

\begin{lemma}[Logarithmic Derivative Residues, \cite{Rosenlicht1972}]\label{lem:rosenlicht}
Let $\phi(u)$ be an algebraic function over $\mathbb{C}(u)$. Then its logarithmic derivative $\phi'(u)/\phi(u)$ is algebraic, and its residue at any point $u_0 \in \mathbb{C}$ is strictly rational: $\operatorname{Res}_{u = u_0} \left({\phi'(u)}/{\phi(u)}\right) \in \mathbb{Q}$.
\end{lemma}

\begin{theorem}[Completeness of Elementary Explicit Solutions]\label{thm:elementary_completeness}
Let $\alpha = A/B \in \mathbb{Q}_{>0}$ with $\gcd(A, B) = 1$. The solution $y(t)$ to the initial value problem \eqref{eq:ode_main} is an elementary function if and only if $\alpha \in \{1, 2\}$.
The corresponding explicit solutions are:
\begin{align}
    \alpha = 1: \quad & y(t) = 1 - \exp(-Ct), \label{eq:sol_1} \\
    \alpha = 2: \quad & y(t) = \tanh(Ct). \label{eq:sol_2}
\end{align}
\end{theorem}

\begin{pf}
Separation of variables establishes $Ct = F_{A/B}(y)$, where $F_{A/B}$ is given by \eqref{eq:decomp_complex} with $u = \phi_1(y) = y^{1/B} \in \mathcal{A}$. For $y(t) = F_{A/B}^{-1}(Ct)$ to be elementary, $F_{A/B}(u)$ must possess an elementary inverse.

First, consider the polynomial component $P(u) = -\sum_{k=1}^{\lfloor (B-1)/A \rfloor} \frac{B u^{B-kA}}{B-kA}$. If $P(u)$ is non-constant, exponentiating $F_{A/B}(u)$ yields an expression $Q(u) \exp(P(u))$ with algebraic $Q(u)$.
According to Theorem~\ref{thm:ritt}, the functional inverse of such an expression cannot be elementary. The product of a non-constant algebraic function $Q(u)$ and an exponential $\exp(P(u))$ fundamentally constitutes a binary algebraic combination, which cannot be reduced to the strict single-variable compositional chain required by Ritt's theorem. Hence, the polynomial component must be absent,
\begin{equation}\label{eq:ritt_deg}
    \left\lfloor \frac{B - 1}{A} \right\rfloor = 0 \implies B \le A.
\end{equation}

Second, we analyze the logarithmic terms. We combine them into a~single logarithm:
\begin{equation}
    \sum_{k=0}^{A-1} c_k \operatorname{Log}(u - \omega_k) = \operatorname{Log}  \prod_{k=0}^{A-1} (u - \omega_k)^{c_k}  = \operatorname{Log}(\phi_2(u)),
\end{equation}
where $\phi_2(u) := \prod_{k=0}^{A-1} (u - \omega_k)^{-(B/A)\omega_k^B}$. By the structural constraints of Theorem~\ref{thm:ritt}, the argument of the transcendental logarithm function must be an algebraic function, establishing $\phi_2 \in \mathcal{A}$.

Taking the logarithmic derivative of $\phi_2(u)$ yields $\frac{\phi_2'(u)}{\phi_2(u)} = -\frac{B}{A} \sum_{k=0}^{A-1} \frac{\omega_k^B}{u - \omega_k}$. By Lemma~\ref{lem:rosenlicht}, the residue at each simple pole $\omega_k$ must be a~rational number:
\begin{equation}
    \operatorname{Res}_{u = \omega_k} \frac{\phi_2'(u)}{\phi_2(u)} = -\frac{B}{A} \omega_k^B \in \mathbb{Q}, \quad \forall k \in \{0, 1, \dots, A-1\}.
\end{equation}
This strictly requires $\omega_k^B \in \mathbb{Q}$ for all $k$. Because $\gcd(A, B) = 1$, the values taken by the sequence $\omega_k^B$ span all the $A$-th roots of unity. The only rational roots of unity are $\pm 1$, thus all $A$-th roots of unity are rational if and only if $A \le 2$. 

Combining this with condition \eqref{eq:ritt_deg} leaves exactly two cases:
\begin{itemize}
    \item If $A = 1$: $B \le 1 \implies B = 1$, giving $\alpha = 1$.
    \item If $A = 2$: $B \le 2$. Since $\gcd(2, B) = 1$, $B = 1$, giving $\alpha = 2$.
\end{itemize}
For $\alpha = 1$, $\int \frac{\mathrm{d}y}{1-y} = -\ln(1-y) = Ct$, yielding \eqref{eq:sol_1}. For $\alpha = 2$, $\int \frac{\mathrm{d}y}{1-y^2} = \operatorname{artanh}(y) = Ct$, yielding \eqref{eq:sol_2}.
\end{pf}

\subsection{Extension to the Lambert $W$ Function}

We broaden the framework to include the Lambert $W$ function, the principal inverse of $w \mapsto w \exp(w)$ \cite{Corless1996}.

\begin{theorem}[Completeness Under Lambert $W$ Inversion]\label{thm:lambert_completeness}
Let $\alpha = A/B \in \mathbb{Q} \setminus \{0\}$ with $\gcd(A, B) = 1$. The solution $y(t)$ to \eqref{eq:ode_main} is expressible in terms of elementary functions and the Lambert $W$ function if and only if $\alpha \in \left\{ -1, \, \frac{1}{2}, \, 1, \, 2 \right\}$. The two non-elementary solutions are:
\begin{align}
    \alpha = -1: \quad & y(t) = 1 + W_0\left( -\frac{1}{\mathrm{e}} \exp(Ct) \right), \label{eq:sol_minus1} \\
    \alpha = \frac{1}{2}: \quad & y(t) = \left( 1 + W_0\left( -\frac{1}{\mathrm{e}} \exp\left(-\frac{1}{2}Ct\right) \right) \right)^2. \label{eq:sol_half}
\end{align}
\end{theorem}

\begin{pf}
Let $|\alpha| = A/B$ with $r = \operatorname{sgn}(\alpha)$. Setting $u = y^{1/B}$, exponentiating the primitive gives $\exp(F_\alpha(u)) = f(u) \exp(g(u)) = \exp(Ct)$, where $g(u)$ is a~polynomial and $f(u) = \prod_{k=0}^{A-1} (u - \omega_k)^{-r \frac{B}{A} \omega_k^B}$.

Solvability via the Lambert $W$ function requires a~change of variables $\Phi(z) = az^b$ such that $\Phi(f(u) \exp(g(u))) = \psi(u) \exp(\psi(u))$. This implies $[f(u)]^b = b g(u) + \ln a$. Since the right-hand side is a~polynomial, $[f(u)]^b$ must also be a~polynomial. Because the exponents of $[f(u)]^b$ sum to $-b r B \, \delta_{A, 1}$, a~necessary condition for polynomiality is $A = 1$. When $A = 1$, this polynomial matching condition simplifies to
\begin{equation}\label{eq:A1_identity}
    (u - 1)^{-b r B} = b u^B \mathbb{1}_{\{r < 0\}} - b r \sum_{k=1}^{B-1} \frac{B u^{B-k}}{B-k} + \ln a.
\end{equation}
Defining integer $n := -b r B$, we obtain 
\begin{equation}(u - 1)^n = -\frac{r n}{B} u^B \mathbb{1}_{\{r < 0\}} + \sum_{k=1}^{B-1} \frac{n u^{B-k}}{B-k} + (-1)^n.\end{equation}
For $\alpha > 0$ ($r = +1$), matching degrees yields $n = B - 1$, which exclusively produces a~valid polynomial identity at $B = 2$ ($n = 1$), yielding $\alpha = 1/2$. For $\alpha < 0$ ($r = -1$), matching degrees requires $B = n$, which is exclusively valid at $B = 1$ ($n=1$), yielding $\alpha = -1$.
\end{pf}

\subsection{Limiting Behavior for $\alpha \to 0$}

\begin{theorem}[Convergence to the Logarithmic Integral]\label{thm:log_limit}
As $\alpha \to 0$, the initial value problem \eqref{eq:ode_main} converges pointwise for all $y \in (0, 1)$ to the singular autonomous equation $\dot{y} = -C \ln y$. Its exact solution is given by $y(t) = \operatorname{li}^{-1}(-Ct)$, where $\operatorname{li}^{-1}$ denotes the inverse of the logarithmic integral $\operatorname{li}(y) = \int_0^y \frac{\mathrm{d}s}{\ln s}$, which is uniquely defined because $\operatorname{li}(y)$ is strictly decreasing on the physical domain $y \in (0, 1)$ \cite{NIST:DLMF}.
\end{theorem}

\begin{pf}
Applying l'H\^opital's rule to the right-hand side of \eqref{eq:ode_main} with respect to $\alpha$ yields
\begin{equation}
    \lim_{\alpha \to 0} \frac{1 - y^\alpha}{\alpha} = -\lim_{\alpha \to 0} \frac{\exp(\alpha \ln y) - 1}{\alpha} = -\ln y.
\end{equation}
Separating variables in this limiting differential equation yields $\int_0^y \frac{\mathrm{d}s}{\ln s} = -Ct$, which perfectly matches the definition of $\operatorname{li}(y)$ on the domain $y \in (0, 1)$.
\end{pf}

\section{Conclusion}

This paper provides a~complete characterization of the analytical invertibility of the nonlinear relaxation model $\dot{y} = C(1 - y^\alpha)$. By applying Ritt's theorem and Rosenlicht's differential algebraic residue criterion to the cyclotomic decomposition, we proved that elementary explicit solutions exist strictly for $\alpha \in \{1, 2\}$. Extending the functional framework using the Lambert $W$ function introduces closed-form inverses for $\alpha \in \{-1, 1/2\}$. As $\alpha \to 0$, the solution space is completed by the logarithmic integral $\operatorname{li}(y)$. This resolves the question of closed-form solvability for this family of equations.

From an applied perspective, identifying these strict algebraic boundaries serves two distinct purposes. First, these exact functional forms provide zero-error, explicit ground-truth solutions for validating numerical integration schemes and asymptotic approximations used in porous media flow simulators \cite{Hairer1993}. Second, for real gas systems operating near these specific polytropic indices, these exact step responses enable the direct algebraic isolation of the lumped permeability parameter $C$ from transient downstream pressure measurements, bypassing the need for computationally expensive non-linear regression over implicit integral representations.

\section*{Declaration of Generative AI in the manuscript preparation process}
During the preparation of this work, \texttt{Gemini 3.8 Flash} was used to improve the language, stylistics, and readability of the manuscript. After using this tool, the author reviewed, edited, and verified the text as needed and takes full responsibility for the final content of the published article.
\printcredits


\end{document}